\documentclass[12pt]{amsart}
\usepackage{amsmath,amsfonts,amssymb,amsthm,a4wide,hyperref}
\usepackage[all]{xy}
\xyoption{tips}

\SelectTips{cm}{12}

\theoremstyle{plain}
\newtheorem{theorem}{Theorem}[section]
\newtheorem*{theorem*}{Theorem}
\newtheorem{proposition}[theorem]{Proposition}
\newtheorem{lemma}[theorem]{Lemma}
\newtheorem{corollary}[theorem]{Corollary}

\DeclareMathOperator{\Hom}{Hom}
\DeclareMathOperator{\Ext}{Ext}
\DeclareMathOperator{\Proj}{\mathsf{Proj}}
\DeclareMathOperator{\thick}{\mathsf{thick}}
\DeclareMathOperator{\fgmod}{\mathsf{mod}}
\DeclareMathOperator{\stmod}{\mathsf{stmod}}
\DeclareMathOperator{\rel}{\mathsf{rel}}
\DeclareMathOperator{\strel}{\mathsf{strel}}
\DeclareMathOperator{\proj}{\mathsf{proj}}
\DeclareMathOperator{\Spc}{Spc}
\DeclareMathOperator{\supp}{supp}

\newcommand{\mathbbold}[1]{\mathbf{#1}}

\begin{document}

\title{The prime spectra of relative stable module categories}

\author{Shawn Baland}
\address{Department of Mathematics,
University of Washington, Box 354350, Seattle, WA 98195}
\email{shawn.baland@gmail.com}

\author{Alexandru Chirvasitu}
\address{Department of Mathematics,
University of Washington, Box 354350, Seattle, WA 98195}
\email{chirva@math.washington.edu}

\author{Greg Stevenson}
\address{Universit\"at Bielefeld, Fakult\"at f\"ur Mathematik,
BIREP Gruppe, Postfach 10\,01\,31, 33501 Bielefeld, Germany}
\email{gstevens@math.uni-bielefeld.de}

\thanks{The first and third authors are grateful to Universit\"at
  Bielefeld and were partially supported by CRC 701 during a portion
  of the period in which this research was conducted. The second
  author was partially supported by the NSF grant DMS-1565226. The third
  author was also partially supported by a fellowship from the Alexander
  von Humboldt Foundation.}

\dedicatory{To Dave Benson on the occasion of his 60th birthday}

\begin{abstract}
For a finite group $G$ and an arbitrary commutative ring $R$,
Brou\'{e} has placed a Frobenius exact structure
on the category of finitely generated $RG$-modules by taking the
exact sequences to be those that split upon restriction to the trivial
subgroup. The corresponding stable category is then tensor triangulated.
In this paper we examine the case $R=S/t^n$, where $S$ is a discrete
valuation ring having uniformising parameter $t$. We prove that the
prime ideal spectrum (in the sense of Balmer) of this `relative' version
of the stable module category of $RG$ is a disjoint union of $n$ copies
of that for $kG$, where $k$ is the residue field of $S$.
\end{abstract}

\maketitle

\tableofcontents

\thispagestyle{empty}

\section{Introduction}
Let $G$ be a finite group and $k$ a field whose characteristic divides
the order of $G$.
One of the main goals in modular representation theory is to try to
understand the stable module category $\stmod kG$ of $kG$.
The objects in $\stmod kG$ are the same as those in the category $\fgmod kG$
of finitely generated $kG$-modules.
If $M$ and $N$ are finitely generated $kG$-modules, then the morphisms
from $M$ to $N$ in $\stmod kG$ are elements of the quotient
\[
\underline{\Hom}_{kG}(M,N)=\Hom_{kG}(M,N)/\mathrm{PHom}_{kG}(M,N),
\]
where $\mathrm{PHom}_{kG}(M,N)$ denotes the set of $kG$-module
homomorphisms $M\rightarrow N$ that factor through some projective
$kG$-module.
In this case $\fgmod kG$ is a \emph{Frobenius category},
and so $\stmod kG$ is triangulated.
The suspension of an object $M$ in $\stmod kG$ is defined to be
its cosyzygy $\Omega^{-1}M$, the cokernel of the inclusion of $M$ into
its injective hull.
A distinguished triangle
\[
\xymatrix{
M' \ar@{>}[r] & M \ar@{>}[r] & M'' \ar@{>}[r] & \Omega^{-1}M'
}
\]
in $\stmod kG$ is, by definition, induced by a short exact sequence
\[
\xymatrix{
0 \ar@{>}[r] & M' \ar@{>}[r] & M \ar@{>}[r] & M'' \ar@{>}[r] & 0
}
\]
in $\fgmod kG$.
Further details may be found in \cite{Happel:1988a}, for example.
Moreover, equipping $-\otimes_k-$ with the diagonal $G$-action
gives an exact symmetric monoidal structure on $\fgmod kG$.
This passes down to $\stmod kG$, making it a
\emph{tensor triangulated category}.

Now suppose that we replace $k$ with an arbitrary commutative ring $R$.
Can one hope to study the category $\fgmod RG$ by mimicking the above setup?
The first obstruction to doing so is the fact that $\fgmod RG$ may no
longer be Frobenius, in which case $\stmod RG$ would fail to be
triangulated in the usual way.
Even if $\fgmod RG$ were Frobenius,
there would still be no guarantee that tensoring over $R$ would be exact,
that is, $\stmod RG$ might not be tensor triangulated.

Brou\'{e} \cite{Broue:2009a} has introduced an alternative exact structure
on $\fgmod RG$, which was based on work of Higman \cite{Higman:1954a}
and later examined by Benson, Iyengar and Krause
\cite{Benson/Iyengar/Krause:2013a}.
Let
\[
\xymatrix{
\iota\colon R \ar@{>}[r] & RG
}
\]
denote the inclusion of the ground ring $R$ and
\[
\xymatrix{
\iota_*\colon \fgmod RG \ar@{>}[r] & \fgmod R
}
\]
the restriction of scalars functor.
An exact sequence of $RG$-modules
\[
\xymatrix{
0 \ar@{>}[r] & M' \ar@{>}[r] & M \ar@{>}[r] & M'' \ar@{>}[r] & 0
}
\]
is defined to be admissible if its restriction
\[
\xymatrix{
0 \ar@{>}[r] & \iota_*M' \ar@{>}[r] & \iota_*M \ar@{>}[r] &
\iota_*M'' \ar@{>}[r] & 0
}
\]
is split exact.
In this paper we will denote the category of finitely generated
$RG$-modules endowed with this `relatively split' exact structure
by $\rel RG$.
As shown in \cite{Broue:2009a}, $\rel RG$ \emph{is} a Frobenius category,
so its stable category $\strel RG$ is triangulated.

Specifically, the injective/projective objects in
$\rel RG$ are the direct summands of those $RG$-modules
lying in the essential image of the induction functor
\[
\xymatrix{
\iota^*=RG\otimes_R-\colon\fgmod R \ar@{>}[r] & \fgmod RG.
}
\]
We shall call such modules
\emph{weakly projective}.
The suspension of an object $M$ in $\strel RG$ is the cokernel $\Sigma M$
in the $R$-split short exact sequence
\[
\xymatrix{
0 \ar@{>}[r] & M \ar@{>}[r] & \iota^*\iota_*M \ar@{>}[r] &
\Sigma M \ar@{>}[r] & 0,
}
\]
where the map $M\rightarrow\iota^*\iota_*M=RG\otimes_R M$ is given by
$m\mapsto\sum_{g\in G}g\otimes g^{-1}m$.
A nice feature of the relative stable category is that,
in the special case where $R=k$ is a field,
$\strel kG$ and $\stmod kG$ coincide.

Another motivation for appealing to this relative exact structure is that
tensoring over $R$ preserves $R$-split short exact sequences,
so $\strel RG$ is tensor triangulated with unit the trivial module $R$.
The relative stable category therefore provides a setting in which one may
exploit a monoidal structure to study representations of $G$ over arbitrary
commuta-
tive rings.

With the relative stable category in mind, we now recall that Balmer
\cite{Balmer:2005a} has developed a general framework for studying
the coarse structures of essentially small tensor triangulated categories
in terms of supports, namely the yoga of \emph{tensor triangular geometry}.
Let $(\mathcal{K},\otimes,\mathbbold{1})$ be an essentially small tensor
triangulated category. A thick subcategory $\mathcal{I}$ of
$\mathcal{K}$ is a \emph{tensor ideal} of $\mathcal{K}$ if
$\mathcal{K}\otimes\mathcal{I}\subseteq\mathcal{I}$.
Continuing the analogy with commutative algebra, Balmer defines
a proper tensor ideal $\mathcal{P}$ of $\mathcal{K}$ to be \emph{prime} if
whenever $x$ and $y$ are objects in $\mathcal{K}$ satisfying
$x\otimes y\in\mathcal{P}$, then $x\in\mathcal{P}$ or $y\in \mathcal{P}$.
The collection of prime ideals of $\mathcal{K}$ is called the
(\emph{prime ideal}) \emph{spectrum} of $\mathcal{K}$,
denoted $\Spc\mathcal{K}$.

For an object $x\in\mathcal{K}$, Balmer defines the \emph{support} of $x$
to be the subset
\[
\supp(x)=\{\mathcal{P}\in\Spc\mathcal{K}\mid x\notin\mathcal{P}\}
\]
of $\Spc\mathcal{K}$.
Such subsets form a basis of closed subsets of a Zariski topology on
$\Spc\mathcal{K}$.
The topological space $\Spc\mathcal{K}$ and the assignment
$x\mapsto\supp(x)$ form something Balmer calls a
\emph{universal support data} for $\mathcal{K}$.
Roughly speaking, this means that $\Spc\mathcal{K}$ is the best
possible abstract setting in which to study support theoretic questions
about the category $\mathcal{K}$.
For example, one may express the universality of $\Spc\mathcal{K}$ by
interpreting it as the space dual to the lattice of (radical) thick
tensor ideals.
Accordingly, any question about the structure of the lattice of
tensor ideals is really a question about $\Spc\mathcal{K}$.

To see all of this in action, we recall a celebrated result of
Benson, Carlson and Rickard \cite{Benson/Carlson/Rickard:1997a},
which states that the thick tensor ideals of $\stmod kG$ are in one to one
correspondence with the specialisation closed subsets of $\Proj H^*(G,k)$,
where $H^*(G,k)$ denotes group cohomology with coefficients in $k$, i.e.,
$\Ext_{kG}^*(k,k)$.
Statements that relate the structure of $\stmod kG$ to the geometry of
$\Proj H^*(G,k)$ occupy the realm of \emph{support theory}.
For example, the above correspondence assigns to each thick tensor ideal
$\mathcal{I}$ in $\stmod kG$ the subset $\bigcup_{M\in\mathcal{I}}V_G(M)$
of $\Proj H^*(G,k)$, where $V_G(M)$ is the \emph{support variety} of $M$.
(See \cite[Chapter 5]{Benson:1991b} for details.) Viewed in Balmer's
framework, this result may be reinterpreted as saying that
$(\Proj H^*(G,k),V_G(-))$ is the \emph{classifying support data} for
$\stmod kG$.
(See \cite[Section 5]{Balmer:2005a}.)
In particular, $\Spc(\stmod kG)$ is homeomorphic to $\Proj H^*(G,k)$.

We remark that little is known about relative stable categories in general.
The goal of this paper is to determine the prime ideal spectrum of
$\strel RG$ in perhaps the most basic non-trivial case,
namely that in which $R$ is the ring $\mathbb{Z}/p^n$,
where $p$ is a prime number and $n\geq 0$.

More generally, the proofs go through for $R=S/t^n$ where $S$ is a
discrete valuation ring with uniformising parameter $t$.
In that context, our main result is the following
(the above case being that of the localisation $S=\mathbb{Z}_{(p)}$
and $t=p$).

\begin{theorem}\label{thm:maintheorem}
  Let $S$ be a discrete valuation ring having residue field $k$ and
  uniformising parameter $t$.
  Setting $R_n=S/t^n$, there is a homeomorphism
\[
\Spc(\strel R_nG)\cong\coprod_{i=1}^n\Spc(\strel kG).
\]
In other words, the prime ideal spectrum of $\strel R_nG$ decomposes into
$n$ disjoint copies of the prime ideal spectrum of $\strel kG=\stmod kG$.
\end{theorem}

As mentioned above, the spectrum of $\stmod kG$ is known,
so the theorem yields a complete description of the spectrum of
$\strel R_nG$.

This computation is also valuable from the point of view of abstract
tensor triangular geometry.
There are many examples in which the spectrum has been computed,
but they all tend to share a common feature---the tensor triangulated
category in question is rigid.
Relative stable categories need not be rigid,
hence they provide a new family of examples that can be fed back into
the abstract theory.
For instance, if the spectrum of a rigid tensor triangulated category is
a disjoint union of subspaces, then the category itself decomposes into
a direct sum of subcategories indexed by those subspaces.
However, in our example the relative stable category is indecomposable.
The information encoded in the triviality of the topology of
the spectrum must therefore manifest in more subtle ways that would be
interesting to explore.

\section{Notation and preliminary calculations}\label{sec.notation}
Let $(S,\mathfrak{m},k)$ be a \emph{discrete valuation ring}, that is,
a local principal ideal domain whose unique maximal ideal is $\mathfrak{m}$
and whose residue field is
\[
k=S/\mathfrak{m}.
\]
(See \cite[Chapter 9]{Atiyah/MacDonald:1969a}.)
We denote by $t$ a \emph{uniformising parameter}, i.e.,
a generator of $\mathfrak{m}$.
For a positive integer $n$ we let
\[
R_n=S/t^n.
\]
Throughout this paper $G$ will denote a finite group.
We set
\[
A_n=R_nG,
\]
the group algebra of $G$ over $R_n$.
Keeping the notation from the introduction, we continue to denote
the inclusion of the ground ring by $\iota\colon R_n\hookrightarrow A_n$.
For each $i\leq n$, the canonical surjection $A_n\rightarrow A_i$ gives
$A_i$ the structure of an $A_n$-module.
We write $\Omega_i$ for the syzygy and $\Omega_i^{-1}$ for the cosyzygy,
taken with respect to the usual abelian structure in $\fgmod A_i$.

We dedicate the remainder of this section to module theoretic computations
that rely on the group structure of $G$.
The rest of the paper depends by and large only on the ring structure of $S$
and may be read almost independently of these results.

We begin with an explicit description of the weakly projective modules
in $\fgmod A_n$, that is, the injective/projective objects in $\rel A_n$.

\begin{proposition}\label{prop:weaklyprojectives}
Every weakly projective $A_n$-module is a direct sum of objects in
\[
\bigcup_{i=1}^n\proj A_i,
\]
where $\proj A_i$ is the full subcategory of finitely generated projective
$A_i$-modules.
\end{proposition}

\begin{proof}
As mentioned in the introduction, the weakly projective $A_n$-modules are
the direct summands of the modules in
\[
\{\iota^*N\mid N\in\fgmod R_n\}.
\]
(Recall that $\iota^*N$ is the induced module $A_n\otimes_{R_n}N$.)
The indecomposable $R_n$-modules are of the form $R_i=S/t^i$ for
$1\leq i\leq n$.
The result follows by noting that $A_n\otimes_{R_n}R_i=A_i$.
\end{proof}

We remark that since each $A_i$ is $R_i$-free,
every object in $\proj A_i$ is also $R_i$-free.

The main object of focus in the sequel will be the cosyzygy $\Omega_n^{-1}k$,
which is the cokernel in the short exact sequence of $A_n$-modules
\begin{equation}\label{eqn:cosyzygyses}
\xymatrix{
0 \ar@{>}[r] & k \ar@{>}[r] & A_n \ar@{>}[r] &
\Omega_n^{-1}k \ar@{>}[r] & 0,
}
\end{equation}
the map $k\rightarrow A_n$ being given by $1\mapsto t^{n-1}\sum_{g\in G}g$.
We first study the behaviour of $\Omega_n^{-1}k$ under base change.

\begin{lemma}\label{lemma:omegabasechange}
$\Omega_n^{-1}k\otimes_{R_n}R_{n-1}\cong A_{n-1}$.
\end{lemma}

\begin{proof}
Since the embedding $k\hookrightarrow A_n$ maps into $t^{n-1}A_n$,
tensoring it with $R_{n-1}$ yields the zero map.
By the right exactness of $-\otimes_{R_n}R_{n-1}$,
the sequence (\ref{eqn:cosyzygyses}) gives rise to the exact sequence
\[
\xymatrix{
k \ar@{>}[r]^-{0} & A_{n-1} \ar@{>}[r] &
\Omega_n^{-1}k\otimes_{R_n}R_{n-1} \ar@{>}[r] & 0.
}
\]
The map $A_{n-1}\rightarrow\Omega_n^{-1}k\otimes_{R_n}R_{n-1}$
is therefore an isomorphism.
\end{proof}

The following ingredient will ensure that a special triangle exists in
$\strel A_n$.

\begin{lemma}\label{lemma:themagicsequence}
There exists an $R_n$-split short exact sequence of $A_n$-modules
\begin{equation}\label{eqn:themagicsequence}
\xymatrix{
0 \ar@{>}[r] & R_{n-1} \ar@{>}[r] & \Omega_n^{-1} k \ar@{>}[r] &
\Omega_n^{-1}R_n \ar@{>}[r] & 0.
}
\end{equation}
\end{lemma}
\begin{proof}
The submodule $\Omega_n k$ of $A_n$ is the collection of elements
$\sum_{g\in G}r_g g$ satisfying
\[
\sum_{g\in G}r_g\in tR_n.
\]
Consider the surjective $A_n$-module homomorphism
$\phi\colon\Omega_n k\rightarrow R_{n-1}$ given by
\[
\sum_{g\in G}r_g g\longmapsto\frac{1}{t}\sum_{g\in G}r_g\pmod{t^{n-1}}.
\]
The kernel of $\phi$ is the collection of elements
$\sum_{g\in G}r_g g$ satisfying $\sum_{g\in G}r_g=0$, which we identify
with $\Omega_nR_n$.
We therefore have a short exact sequence
\[
\xymatrix{
0 \ar@{>}[r] & \Omega_n R_n \ar@{>}[r] & \Omega_n k \ar@{>}[r] &
R_{n-1} \ar@{>}[r] & 0.
}
\]
Note that $\Omega_n R_n$ is injective as an $R_n$-module,
so this sequence is $R_n$-split.
The sequence (\ref{eqn:themagicsequence}) is obtained by applying
$\Hom_{R_n}(-,R_n)$, which preserves $R_n$-split exact sequences.
\end{proof}

Our final result of the section will be used later in Section \ref{sec.cyclic}
to establish certain orthogonality relations.

\begin{lemma}\label{lemma:cosyzygytensors}
For all $1\leq i< j\leq n$ we have
$\Omega_i^{-1}k\otimes_{R_n}\Omega_j^{-1}k\cong
A_{i-1}^{\vphantom{\oplus(|G|-1)}}\oplus A_i^{\oplus(|G|-1)}$.
\end{lemma}

\begin{proof}
Because $t^j$ annihilates both $R_i$ and $R_j$, we have
\[
\Omega_i^{-1}k\otimes_{R_n}\Omega_j^{-1}k\cong
\Omega_i^{-1}k\otimes_{R_j}\Omega_j^{-1}k.
\]
We may therefore compute the right hand term in $\fgmod A_j$.
Writing $f=\sum_{g\in G}g\in A_j$ and applying
$-\otimes_{R_j}\Omega_i^{-1}k$ to the short exact sequence
(\ref{eqn:cosyzygyses}) yields the exact sequence
\[
\xymatrix @C=0.6in{ k\otimes_{R_j}\Omega_i^{-1}k
\ar@{>}[r]^{t^{j-1}f\otimes 1\ } &
A_j\otimes_{R_j}\Omega_i^{-1}k \ar@{>}[r] &
\Omega_j^{-1}k\otimes_{R_j}\Omega_i^{-1}k \ar@{>}[r] & 0 }
\]
so that $\Omega_j^{-1}k\otimes_{R_j}\Omega_i^{-1}k$ is isomorphic to
$A_j\otimes_{R_j}\Omega_i^{-1}k$ modulo the image of $t^{j-1}f\otimes 1$.
By Frobenius reciprocity and the sequence (\ref{eqn:themagicsequence}),
one sees that
\begin{equation}
A_j\otimes_{R_j}\Omega_i^{-1}k\cong
A_{i-1}^{\vphantom{\oplus(|G|-1)}}\oplus A_i^{\oplus(|G|-1)}.
\end{equation}
Since $i<j$, $t^{j-1}$ annihilates the right hand term, so the
image of $t^{j-1}f\otimes 1$ is zero and
\[
\Omega_j^{-1}k\otimes_{R_j}\Omega_i^{-1}k\cong
A_{i-1}^{\vphantom{\oplus(|G|-1)}}\oplus A_i^{\oplus(|G|-1)}. \qedhere
\]
\end{proof}

\section{Localisation sequences and splitting of spectra}\label{sec:localisation}
To begin our discussion of triangulated categories, we now set
\[
\mathcal{D}_n=\strel A_n,
\]
the relative stable module category of $A_n$.
As noted in the introduction, this is a tensor triangulated category
with tensor unit
\[
\mathbbold{1}_n=R_n.
\]
As is customary, we denote the suspension in $\mathcal{D}_n$ by $\Sigma$,
keeping in mind that $\Sigma$ is not the cosyzygy $\Omega_n^{-1}$ in general.

Observe that it is not necessary to specify in which category $\Sigma$
is suspension.
Indeed, if $i\leq n$, then the canonical surjection $A_n\rightarrow A_i$
induces a fully faithful embedding $\mathcal{D}_i\subseteq\mathcal{D}_n$,
because a short exact sequence of $A_i$-modules is $R_i$-split if and only if
it is $R_n$-split.
Thus the suspension of an object in $\mathcal{D}_i$ will equal that in
$\mathcal{D}_n$.

An essential ingredient in proving Theorem \ref{thm:maintheorem} will be
the notion of a semi-orthogonal decomposition of a tensor triangulated category.
We recall that a \emph{localisation sequence} or \emph{semi-orthogonal decomposition} is
a diagram of exact functors
\[
\xymatrix{
\mathcal{R} \ar[r]<0.5ex>^-{\psi_*} \ar@{<-}[r]<-0.5ex>_-{\psi^!} &
\mathcal{S} \ar[r]<0.5ex>^-{\phi^*} \ar@{<-}[r]<-0.5ex>_-{\phi_*} &
\mathcal{T}
}
\]
in which $\psi^!$ is right adjoint to $\psi_*$ and $\phi_*$ is right adjoint
to $\phi^*$,
the functors $\psi_*$ and $\phi_*$ are fully faithful and there are equalities
\[
\phi_*\mathcal{T} = (\psi_*\mathcal{R})^\perp = \{x\in \mathcal{S} \mid
\Hom_\mathcal{S}(\psi_*\mathcal{R}, x) = 0\}
\]
and
\[
\psi_*\mathcal{R} = {}^\perp(\phi_*\mathcal{T}) = \{x\in \mathcal{S} \mid
\Hom_\mathcal{S}(x, \phi_*\mathcal{T}) = 0\}.
\]

For our purposes we also require that $\mathcal{S}$ is an essentially small
tensor triangulated category and that $\mathcal{R}$ and $\mathcal{T}$ are
tensor ideals of $\mathcal{S}$ under $\psi_*$ and $\phi_*$, respectively.
In this very special situation one also obtains a decomposition of the
prime ideal spectrum of $\mathcal{S}$.
If $\mathcal{C}$ is any subcategory of $\mathcal{S}$,
we define the \emph{support} of $\mathcal{C}$ to be the specialisation
closed subset
\[
\supp_{\mathcal{S}}\mathcal{C}=\bigcup_{x\in\mathcal{C}}\supp_{\mathcal{S}}x
\]
of $\Spc\mathcal{S}$.

\begin{theorem}[\cite{Benson/Iyengar/Krause:2013a}, Theorem A.5]
\label{thm:decomposition}
The subsets $\Spc\mathcal{R}=\supp_{\mathcal{S}}\mathcal{R}$ and
$\Spc\mathcal{T}=\supp_{\mathcal{S}}\mathcal{T}$ of $\Spc\mathcal{S}$
are open and closed, and there is a decomposition
\[
\Spc\mathcal{S}=\Spc\mathcal{R}\coprod\Spc\mathcal{T}.
\]
\end{theorem}

We now return to the categories
$\mathcal{D}_i=\strel A_i$, $i\geq 1$. As explained in
\cite[{Remark 6.10}]{Benson/Iyengar/Krause:2013a}, each canonical
ring epimorphism $\phi_i\colon A_{i} \to A_{i-1}$ induces a
localisation sequence
\[
\xymatrix{
\mathcal{K}_i \ar[r]<0.5ex>^-{{\psi_i}_*} \ar@{<-}[r]<-0.5ex>_-{{\psi_i}^!} &
\mathcal{D}_i \ar[r]<0.5ex>^-{{\phi_i}^*} \ar@{<-}[r]<-0.5ex>_-{{\phi_i}_*} &
\mathcal{D}_{i-1}
}
\]
where ${\phi_i}_*$ is restriction of scalars,
${\phi_i}^* = -\otimes\mathbbold{1}_{i-1} = -\otimes_{R_i}R_{i-1}$
and
\begin{equation}\label{eqn:kernel}
  \mathcal{K}_i = \ker(-\otimes\mathbbold{1}_{i-1}). 
\end{equation}
Note that the functors ${\phi_i}^*$ and ${\psi_i}^!$ are strong monoidal,
i.e., they preserve tensor products and units.
(See \cite[{Chapter XI.2}]{MacLane:1971a} for full details on strong monoidal
functors.) 
It follows that ${\psi_i}_*\mathcal{K}_i$ and ${\phi_i}_*\mathcal{D}_{i-1}$
are thick tensor ideals in $\mathcal{D}_i$.
We frequently identify $\mathcal{K}_i$ and $\mathcal{D}_{i-1}$ with their
images under these fully faithful embeddings.
We also remark in passing that $\mathcal{D}_1$ is just the usual stable module
category $\stmod kG$ of $kG$.
A slight generalisation of the conclusion of
\cite[{Remark 6.10}]{Benson/Iyengar/Krause:2013a} is the following result.

\begin{theorem}\label{thm:splitting}
Setting $\mathcal{K}_1=\mathcal{D}_1$, there is a decomposition of spectra
\[
\Spc\mathcal{D}_n=\coprod_{i=1}^n\Spc\mathcal{K}_i.
\]
\end{theorem}

\begin{proof}
This follows by induction on $n$ and the appropriate use of
Theorem \ref{thm:decomposition}.
\end{proof}

In order to prove Theorem \ref{thm:maintheorem}, it therefore suffices to
show that each $\Spc\mathcal{K}_i$ is homeo-
morphic to $\Spc\mathcal{D}_1$.
To establish this fact, we first show that each $\mathcal{K}_i$ is equal to
the thick tensor ideal
\[
\thick_{\mathcal{D}_i}^\otimes(\Omega_i^{-1}k)
\]
of $\mathcal{D}_i$ generated by $\Omega_i^{-1}k$.
We then exhibit a monoidal equivalence between
$\thick^\otimes_{\mathcal{D}_i}(\Omega_i^{-1}k)$ and $\mathcal{D}_1$.

\section{The tensor ideal $\thick_{\mathcal{D}_n}^\otimes(W_n)$}\label{sec.thick} 
To simplify notation somewhat and to emphasise its central role
throughout the rest of the paper, we now set
\[
W_n=\Omega_n^{-1}k.
\]

\begin{lemma}\label{lemma:kernelgenerated}
We have $\mathcal{K}_n=\thick_{\mathcal{D}_n}^\otimes(W_n)$,
where $\mathcal{K}_n$ is as defined by (\ref{eqn:kernel}).
\end{lemma}

\begin{proof}
By Lemma \ref{lemma:omegabasechange} we have $W_n\otimes\mathbbold{1}_{n-1}=0$
in $\mathcal{D}_n$ so that
$\thick_{\mathcal{D}_n}^\otimes(W_n)\subseteq\mathcal{K}_n$.

Conversely, let $X\in\mathcal{K}_n$. The $R_n$-split short exact sequence
(\ref{eqn:themagicsequence}) induces a triangle
\[
\xymatrix{
\mathbbold{1}_{n-1} \ar@{>}[r] & W_n \ar@{>}[r] &
\Sigma\mathbbold{1}_n \ar@{>}[r] &
}
\]
in $\mathcal{D}_n$.
Since $X\in\mathcal{K}_n$, tensoring this triangle with $X$ yields
a triangle of the form
\[
  \xymatrix{ 0 \ar@{>}[r] & X\otimes W_n \ar@{>}[r] & \Sigma X
    \ar@{>}[r] & }.
\]
It follows that $\Sigma X\cong X\otimes W_n$ so that
$X\in\thick_{\mathcal{D}_n}^\otimes(W_n)$.
\end{proof}

\begin{lemma}\label{lemma:abeliangroupdecomp}
  Every object in $\mathcal{K}_n$ is isomorphic in $\mathcal{D}_n$ to
  an $A_n$-module whose $S$-module decomposition contains only summands
  of the form $R_n$ and $R_{n-1}$.
\end{lemma}

\begin{proof}
  Let $X\in\mathcal{K}_n$.
  Viewing $X$ as an $A_n$-module, write
  $\iota_*X\cong\bigoplus_{i=1}^n R_i^{\oplus r_i}$.
  Let
\[
U=R_n^{\oplus r_n}\oplus R_{n-1}^{\oplus r_{n-1}}\qquad\text{and}\qquad
V=\bigoplus_{i=1}^{n-2} R_i^{\oplus r_i}
\]
  so that $\iota_*X=U\oplus V$.
  Note that $t^{n-1}X\subseteq U$ since $t^{n-1}$ annihilates $V$.
  We thus have
\[
\iota_*(X\otimes_{R_n}R_{n-1})=(U/t^{n-1}X)\oplus V.
\]
By assumption, $X\otimes\mathbbold{1}_{n-1}=0$ in $\mathcal{D}_n$, so
$X\otimes_{R_n}R_{n-1}$ is weakly projective.
Proposition \ref{prop:weaklyprojectives} implies that
\[
  X\otimes_{R_n}R_{n-1}\cong Y\oplus Z,
\]
where $Y\in\proj A_{n-1}$ and $Z$ is a direct sum of objects in
$\bigcup_{i=1}^{n-2}\proj A_i$.
Comparing $S$-mod-
ule structures, we must have
\[
\iota_*Y\cong U/t^{n-1}X\qquad\text{and}\qquad\iota_*Z\cong V.
\]
These $S$-module isomorphisms allow us to place $A_n$-module structures
$\widetilde{U/t^{n-1}X}$ and $\widetilde{V}$ on $U/t^{n-1}X$ and $V$,
respectively, through which
$X\otimes_{R_n}R_{n-1}=\widetilde{U/t^{n-1}X}\oplus\widetilde{V}$.

Now consider the short exact sequence of $A_n$-modules
\begin{equation}\label{eqn:splitsequenceforkernel}
\xymatrix{
0 \ar@{>}[r] & X' \ar@{>}[r] & X \ar@{>}[r] & \widetilde{V} \ar@{>}[r] & 0,
}
\end{equation}
where the right hand map is the composition
\[
\xymatrix{
X \ar@{>}[r]^{\pi\ \ \ \ \ \ \ } &
X\otimes_{R_n}R_{n-1} \ar@{>}[r]^{\ \ \ \ \ \ \mathrm{pr}_2} & \widetilde{V}.
}
\]
Observe that we have $X'=\{m\in X\mid\pi(m)\in\ker\mathrm{pr}_2\}=U$
as abelian groups.
Applying $\iota_*$ to the sequence (\ref{eqn:splitsequenceforkernel})
therefore yields a split short exact sequence of $S$-modules
\[
\xymatrix{
0 \ar@{>}[r] & U \ar@{>}[r] & \iota_*X \ar@{>}[r] & V \ar@{>}[r] & 0.
}
\]
This means that (\ref{eqn:splitsequenceforkernel}) gives rise to a triangle
\[
\xymatrix{
X' \ar@{>}[r] & X \ar@{>}[r] & \widetilde{V} \ar@{>}[r] &
}
\]
in $\mathcal{D}_n$.
Because $\widetilde{V}\cong Z$ is weakly projective in $\fgmod A_n$,
we have $\widetilde{V}=0$ in $\mathcal{D}_n$.
It follows that $X'\cong X$ in $\mathcal{D}_n$.
\end{proof}

\begin{lemma}\label{lemma:syzygykernel}
We have $\Omega_n\mathcal{K}_n=\mathcal{D}_1$.
\end{lemma}

\begin{proof}
Let $X\in\mathcal{K}_n$.
To prove that $\Omega_n X\in\mathcal{D}_1$, it suffices to show that
$t$ annihilates $\Omega_n X$.
Employing Lemma \ref{lemma:abeliangroupdecomp}, we assume that the
  $S$-module decomposition of $\iota_*X$ is
  $R_n^{\oplus r}\oplus R_{n-1}^{\oplus s}$ for some non-negative integers
  $r$ and $s$.
We then have
\begin{equation}\label{eqn:basechangegroupdecomp}
\iota_*(X\otimes_{R_n}R_{n-1})\cong R_{n-1}^{\oplus(r+s)}.
\end{equation}
On the other hand, $X\otimes_{R_n}R_{n-1}$ is known to be weakly projective.
The decomposition (\ref{eqn:basechangegroupdecomp}) implies that
$X\otimes_{R_n}R_{n-1}\in\proj A_{n-1}$.
We thus have $X\otimes_{R_n}R_{n-1}\cong\bigoplus_j A_{n-1}e_j$
for some idempotents $e_j\in A_{n-1}$.
It is known (see \cite[Theorem 21.28]{Lam:2001a}) that since 
the extension of scalars $A_n\rightarrow A_{n-1}$ is a surjective
ring homomorphism with nilpotent kernel,
the $e_j$ lift to idempotents $f_j\in A_n$.

Now consider the projective module $Y=\bigoplus_j A_nf_j$. Let
\[
\xymatrix{
\pi\colon X \ar@{>}[r] & X\otimes_{R_n}R_{n-1}
}
\qquad\text{and}\qquad
\xymatrix{
\phi\colon Y \ar@{>}[r] & X\otimes_{R_n}R_{n-1}
}
\]
be the $A_n$-module homomorphisms induced by the extension of scalars
from $R_n$ to $R_{n-1}$.
Because $\pi$ is surjective and $Y$ is projective, $\phi$ lifts to an
$A_n$-module homomorphism
\[
\xymatrix{
\psi\colon Y \ar@{>}[r] & X
}
\]
satisfying $\pi\circ\psi=\phi$. Note that $\psi\otimes_{R_n}R_{n-1}$
is the identity on $X\otimes_{R_n}R_{n-1}$; in particular, it is surjective.
It follows by Nakayama's lemma that $\psi$ is surjective.
We therefore have a short exact sequence
\[
\xymatrix{
0 \ar@{>}[r] & \Omega_n X \ar@{>}[r] & Y \ar@{>}[r]^\psi &
X \ar@{>}[r] & 0.
}
\]
But $\Omega_n X$ is also contained in the kernel of the composition
$\pi\circ\psi=\phi$.
The kernel of the latter is annihilated by $t$, hence the same is true of
its submodule $\Omega_n X$.

Conversely, let $X\in\mathcal{D}_1$ and consider the short exact sequence
\[
\xymatrix{
0 \ar@{>}[r] & X \ar@{>}[r]^-\phi & A_n^{\oplus r} \ar@{>}[r]^-\psi &
\Omega^{-1}_nX \ar@{>}[r] & 0
}
\]
defining a cosyzygy of $X$ in $\fgmod A_n$.
Since $X$ lies in $\mathcal{D}_1$, we know that $\iota_*X$ is
a $k$-vector space, so the image of $\phi$ is contained in
$t^{n-1}A_n^{\oplus r}$.
This means that $\phi\otimes_{R_n}R_{n-1}=0$,
hence $\psi\otimes_{R_n}R_{n-1}$ is an isomorphism.
We therefore have
$\Omega^{-1}_nX\otimes_{R_n}R_{n-1}\cong A_{n-1}^{\oplus r}$,
and the latter is weakly projective.
This shows that $\Omega^{-1}_nX$ lies in $\mathcal{K}_n$ as claimed.
\end{proof}

We are now ready for the main theorem of this section.

\begin{theorem}\label{thm:thicksameastensor}
  The syzygy $\Omega_n$ induces an equivalence of triangulated categories
\[
  \thick^\otimes_{\mathcal{D}_n}(W_n)\cong\thick^\otimes_{\mathcal{D}_n}(k)=
  \stmod kG.
\]
\end{theorem}

\begin{proof}
  We saw in Lemma \ref{lemma:kernelgenerated} that
  $\thick^\otimes_{\mathcal{D}_n}(W_n)=\mathcal{K}_n$, and we know that
  $\thick^\otimes_{\mathcal{D}_n}(k)=\mathcal{D}_1$.
  Notice that $\Omega_n$ induces a (not necessarily monoidal) exact
  autoequivalence of $\mathcal{D}_n$.
  Indeed, it is straightforward to check that $\Omega_n$ preserves
  $R_n$-split short exact sequences in $\fgmod A_n$.
  The same is then true of its quasi-inverse $\Omega_n^{-1}$.
  The restriction of $\Omega_n$ to $\mathcal{K}_n$ is therefore
  an equivalence onto its essential image.
  By Lemma \ref{lemma:syzygykernel}, the latter is precisely $\mathcal{D}_1$.
\end{proof}

Although the above equivalence is \emph{not} monoidal in general,
the next section will show that $\thick^\otimes_{\mathcal{D}_n}(W_n)$
and $\stmod kG$ are in fact equivalent as \emph{tensor triangulated} categories.

\section{A monoidal equivalence}\label{sec.monoidal}
By the localisation sequences in Section \ref{sec:localisation},
we know that $\thick^\otimes_{\mathcal{D}_n}(W_n)=\mathcal{K}_n$ is equal to
$\mathcal{D}_n/\mathcal{D}_{n-1}$ as tensor ideals in $\mathcal{D}_n$,
so $\thick^\otimes_{\mathcal{D}_n}(W_n)$ is tensor triangulated.
In this section we exhibit a monoidal exact equivalence between
$\thick^\otimes_{\mathcal{D}_n}(W_n)$ and $\mathcal{D}_1$,
namely the restriction of the functor
\[
\xymatrix{
P\colon\mathcal{D}_n \ar@{>}[r] & \mathcal{D}_1
}
\]
induced by multiplication by $t^{n-1}$.

Specifically, if $M$ is an $A_n$-module, then as an abelian group,
$P(M)$ is defined to be the $A_n$-submodule $t^{n-1}M$ of $M$.
Identifying $A_1$ with $A_n/t$, there is an action of $A_1$ on $t^{n-1}M$
given by
\[
  \overline{a}(t^{n-1}m)=t^{n-1}am\qquad\text{for all $a\in A_n$.}
\]
This action is well defined since $t$ annihilates $t^{n-1}M$.
If $\phi\colon M\rightarrow N$ is a homomorphism of $A_n$-modules,
then so is $\phi|_{t^{n-1}M}\colon t^{n-1}M\rightarrow N$.
Moreover, for $m\in M$ we have
\[
\phi(t^{n-1}m)=t^{n-1}\phi(m)\in t^{n-1}N,
\]
hence $\phi|_{t^{n-1}M}$ induces a map $t^{n-1}M\rightarrow t^{n-1}N$.
We therefore set $P(\phi)=\phi|_{t^{n-1}M}$.

Note that multiplication by $t^{n-1}$ preserves $R_n$-split
short exact sequences in $\fgmod A_n$, so $P$ is exact.
We claim that $P$ is monoidal.
To see this, let $M$ and $M'$ be $A_n$-modules and consider the map
\[
\xymatrix{
\phi\colon P(M)\otimes_{R_1} P(M') \ar@{>}[r] & P(M\otimes_{R_n}M')
}
\]
given by $t^{n-1}m\otimes t^{n-1}m'\mapsto t^{n-1}(m\otimes m')$.
One readily checks that $\phi$ is a well defined iso-
morphism of
abelian groups and that the action of $G$ commutes with $\phi$.

Having established that $P$ is an exact tensor functor, we now let
\[
\xymatrix{
\widetilde{P}\colon\thick^\otimes_{\mathcal{D}_n}(W_n) \ar@{>}[r] &
\mathcal{D}_1
}
\]
denote the restriction of $P$ to $\thick^\otimes_{\mathcal{D}_n}(W_n)$.
Our strategy in proving that $\widetilde{P}$ is a monoidal equivalence
will be to show that the functor
\[
\xymatrix{
F=\Omega_n^{-1}\Omega_1^{\vphantom{n}}\colon\mathcal{D}_1 \ar@{>}[r] &
\thick^\otimes_{\mathcal{D}_n}(W_n).
}
\]
is a quasi-inverse.

\begin{lemma}\label{lemma:surjectivefullmonoidal}
The composition $PF\colon\mathcal{D}_1\rightarrow\mathcal{D}_1$ is
naturally isomorphic to the identity functor.
\end{lemma}

\begin{proof}
Let $X$ be an object in $\mathcal{D}_1$ and consider the short exact
sequence of $A_1$-modules
\[
\xymatrix{
0 \ar@{>}[r] & \Omega_1 X \ar@{>}[r] & A_1^{\oplus r} \ar@{>}[r] &
X \ar@{>}[r] & 0
}
\]
defining a syzygy $\Omega_1 X$ in $\fgmod A_1$.
A cosyzygy of $\Omega_1 X$ in $\fgmod A_n$ is then obtained via the
short exact sequence
\[
\xymatrix{
0 \ar@{>}[r] & \Omega_1 X \ar@{>}[r] & A_n^{\oplus r} \ar@{>}[r] &
\Omega_n^{-1}\Omega_1^{\vphantom{n}} X \ar@{>}[r] & 0
}
\]
and we have a diagram
\[
\xymatrix{
0 \ar@{>}[r] & \Omega_1 X \ar@{>}[r] \ar@{=}[d] &
A_1^{\oplus r} \ar@{>}[r] \ar@{>}[d] &
X \ar@{>}[r] \ar@{>}[d] & 0 \\
0 \ar@{>}[r] & \Omega_1 X \ar@{>}[r] & A_n^{\oplus r} \ar@{>}[r] &
FX \ar@{>}[r] & 0,
}
\]
where the right two vertical arrows are those induced by multiplication
by $t^{n-1}$.
The right hand arrow therefore identifies $X$ with the submodule $PFX$.
\end{proof}

\begin{corollary}\label{cor:surjectivefullmonoidal}
  The functor $\widetilde{P}$ is full, essentially surjective and monoidal.
\end{corollary}

\begin{proof}
  The fact that $\widetilde{P}$ is full and essentially surjective
  follows from Lemmas \ref{lemma:surjectivefullmonoidal} and
  \ref{lemma:syzygykernel}, the latter of which implies that the
  essential image of $F$ is contained in $\thick^\otimes_{\mathcal{D}_n}(W_n)$.
  
  By the discussion preceding Lemma \ref{lemma:surjectivefullmonoidal}
  we know that $P\colon\mathcal{D}_n\rightarrow\mathcal{D}_1$ is monoidal,
  so its restriction
  $\widetilde{P}\colon\thick^\otimes_{\mathcal{D}_n}(W_n)\rightarrow\mathcal{D}_1$
  respects tensor products.
  Any such functor that is also essentially surjective will automatically
  be monoidal.
\end{proof}

\begin{lemma}\label{lemma:zerokernel}
The kernel of $\widetilde{P}$ is trivial.
\end{lemma}

\begin{proof}
Let $X$ be an object in $\thick^\otimes_{\mathcal{D}_n}(W_n)$ with
$\widetilde{P}X$ weakly projective.
By Lemma \ref{lemma:abeliangroupdecomp}, we may assume that
$\iota_*X\cong R_n^{\oplus r}\oplus R_{n-1}^{\oplus s}$
for some non-negative integers $r$ and $s$.
Because $\widetilde{P}X$ is weakly projective, we have
$\widetilde{P}X\cong\bigoplus_j A_1e_j$ for some idempotents $e_j$ in $A_1$.

The surjection $A_n\rightarrow A_1$ given by multiplication by
$t^{n-1}$ has kernel $tA_n$, thus it is iso-
morphic to the base change
homomorphism $A_n\rightarrow A_n\otimes_{R_n}k$.
As in the proof of Lemma \ref{lemma:syzygykernel},
there then exist idempotents $f_j$ in $A_n$ satisfying $e_j=t^{n-1}f_j$.
Letting $Y=\bigoplus_j A_nf_j$, we obtain a natural embedding
$\phi\colon \widetilde{P}X\hookrightarrow Y$ mapping $\widetilde{P}X$ isomorphically onto $t^{n-1}Y$.

Note that since $A_n$ is injective as a module over itself and $Y$ is
a direct summand of a free $A_n$-module, $Y$ is also injective.
(Actually, $Y$ is the injective hull of $\widetilde{P}X$.)
This, along with the embedding $\widetilde{P}X\hookrightarrow X$,
allows us to extend $\phi$ to a morphism $\psi\colon X\rightarrow Y$.

Now let $\widetilde{\psi}$ denote the map of free $R_n$-modules
obtained by restricting $\psi$ to the $R_n$-free component
$R_n^{\oplus r}$ of $X$.
Then $t^{n-1}\widetilde{\psi}=\phi$, hence $\widetilde{\psi}$ is an
isomorphism on socles.
This shows that $Y$ has rank $r$ as a free $R_n$-module and that $\psi$
is surjective.
Because $Y$ is projective, we may therefore split off a direct summand
$Y$ from $X$ and assume that $\iota_*X\cong R_{n-1}^{\oplus s}$.
We then have $X\otimes_{R_n}R_{n-1}\cong X$.
But $X$ lies in $\thick^\otimes_{\mathcal{D}_n}(W_n)$, the kernel of
$-\otimes_{R_n}R_{n-1}$, hence $X\cong 0$ in $\mathcal{D}_n$.
\end{proof}

\begin{theorem}\label{thm:equivalence}
  The functor $\widetilde{P}$ is a monoidal equivalence of triangulated
  categories.
\end{theorem}

\begin{proof}
  We know that $\widetilde{P}$ is full and essentially surjective
  by Corollary \ref{cor:surjectivefullmonoidal}, and it has trivial kernel
  by Lemma \ref{lemma:zerokernel}.
  Appealing to a bit of folklore (see \cite[Proposition 3.18]{Balmer:2007a}),
  these con-
  ditions are sufficient for $\widetilde{P}$ to be an equivalence.
  It is monoidal by Corollary \ref{cor:surjectivefullmonoidal}.
\end{proof}

We are now in a position to prove the main result.

\begin{theorem*}[\ref{thm:maintheorem}]
  For every positive integer $n$ there is a homeomorphism
\[
  \Spc\mathcal{D}_n\cong\coprod_{i=1}^n\Spc\mathcal{D}_1.
\]
\end{theorem*}

\begin{proof}
  Setting $\mathcal{K}_1=\mathcal{D}_1$, Theorem \ref{thm:splitting}
  tells us that there is a decomposition of spectra
\[
\Spc\mathcal{D}_n=\coprod_{i=1}^n\Spc\mathcal{K}_i.
\]
By Lemma \ref{lemma:kernelgenerated} we have
$\mathcal{K}_i=\thick_{\mathcal{D}_i}^\otimes(W_i)$.
Theorem \ref{thm:equivalence} shows that the functor
\[
\xymatrix{
\widetilde{P}\colon\thick^\otimes_{\mathcal{D}_i}(W_i) \ar@{>}[r] &
\mathcal{D}_1
}
\]
induced by multiplication by $t^{i-1}$ is an equivalence of tensor
triangulated categories. Put-
ting this all together, we have
$\Spc\mathcal{K}_i\cong\Spc\mathcal{D}_1$ for all $1\leq i\leq n$.
\end{proof}

The following corollary summarises the consequences of our results for
$\mathcal{D}_n$.

\begin{corollary}\label{cor:decomposition}
  The relative stable module category $\mathcal{D}_n=\strel R_nG$
  admits a semi-orthogo-
  nal decomposition into $n$ tensor ideals
\[
  \strel R_nG = (\stmod kG, \ldots, \stmod kG),
\]
where the $i$th copy embeds as
$\mathcal{K}_i=\thick^\otimes_{\mathcal{D}_n}(W_i)$.
\end{corollary}

\begin{proof}
This follows from the discussion in Section \ref{sec:localisation},
along with Theorem \ref{thm:equivalence} and Lemma \ref{lemma:kernelgenerated}.
\end{proof}

\section{An example: cyclic groups of prime order}\label{sec.cyclic}
In this section we provide an explicit description of the spectrum
of $\strel R_nG$ in the case where the residue field $k$ has
prime characteristic $p$ and $G = C_p$, the cyclic group of order $p$.
In particular, we give concrete generators for all of the prime tensor ideals.

The first several results actually hold for any finite group $G$.
We remind the reader that $W_i$ denotes the cosyzygy $\Omega_i^{-1}k$,
taken with respect to the usual abelian category structure in
$\fgmod A_i=\fgmod R_iG$.

\begin{proposition}\label{prop:omegasgenerate}
For any finite group $G$, the $A_n$-modules $W_i$
for $i\leq n$ generate $\mathcal{D}_n=\strel A_n$ as a thick tensor ideal.
Any prime tensor ideal in $\Spc\mathcal{D}_n$ contains at least $n-1$
objects in the set $\{W_1,\ldots,W_n\}$.
\end{proposition}

\begin{proof}
The first statement follows directly from Corollary \ref{cor:decomposition}.
For the second statement, let $\mathcal{P}\in\Spc\mathcal{D}_n$ and
suppose that $W_i\notin\mathcal{P}$.
By Lemma \ref{lemma:cosyzygytensors} we have
\[
W_i\otimes W_j=0\in\mathcal{P}.
\]
Since $\mathcal{P}$ is prime, this shows that $W_j\in\mathcal{P}$ for all
$j\neq i$.
\end{proof}

Motivated by the previous lemma, we now focus our attention on certain
thick tensor ideals of $\mathcal{D}_n$.
For $1\leq i\leq n$, we let
\[
\mathcal{P}_{i,n}=
\thick_{\mathcal{D}_n}^\otimes(\{W_1,\ldots,W_n\}\setminus\{W_i\}).
\]
Our goal will be to show that these are precisely the prime tensor ideals in
$\mathcal{D}_n$ in the case where $G$ is the cyclic group $C_p$.
(We know from Theorem \ref{thm:maintheorem} that the spectrum will be a disjoint
union of $n$ points.)
The following two results hold for any finite group $G$.

\begin{lemma}\label{lemma:imageofbasechange}
For all $X\in\mathcal{P}_{i,n}$ we have
$X\otimes\mathbbold{1}_{n-1}\in\mathcal{P}_{i,n-1}$, i.e., 
\[
{\phi_n}^*\mathcal{P}_{i,n}\subseteq\mathcal{P}_{i,n-1}.
\]
\end{lemma}

\begin{proof}
If $i\leq n-1$ then $W_i\otimes\mathbbold{1}_{n-1}=W_i$, whereas
$W_n\otimes\mathbbold{1}_{n-1}=0$ by Lemma \ref{lemma:omegabasechange}.
Hence ${\phi_n}^*$ sends the generators of $\mathcal{P}_{i,n}$ into
$\mathcal{P}_{i,n-1}$.
Because $\mathcal{P}_{i,n-1}$ is thick and ${\phi_n}^*$ is exact,
the lemma follows immediately.
(The dubious reader may consult \cite[Lemma 3.8]{Stevenson:2013a}.)
\end{proof}

\begin{lemma}\label{lemma:idealsareproper}
Each tensor ideal $\mathcal{P}_{i,n}$ is proper in $\mathcal{D}_n$.
\end{lemma}

\begin{proof}
We fix $i$ and proceed by induction on $n$. For the base $n=i$ we need to
show that
\[
\mathcal{P}_{i,i}=\thick_{\mathcal{D}_i}^\otimes(W_1,\ldots,W_{i-1})
\]
is proper.
We saw in Section \ref{sec:localisation} that the restriction of scalars
${\phi_i}_*$ embeds $\mathcal{D}_{i-1}$ as a proper tensor ideal in
$\mathcal{D}_i$.
For each $1\leq j \leq i-1$ we have $W_j\in{\phi_i}_*\mathcal{D}_{i-1}$,
so $\mathcal{P}_{i,i}$ is contained in ${\phi_i}_*\mathcal{D}_{i-1}$
and $\mathcal{P}_{i,i}$ is proper in $\mathcal{D}_i$.
(In fact, $\mathcal{P}_{i,i}={\phi_i}_*\mathcal{D}_{i-1}$ by Proposition
\ref{prop:omegasgenerate}.)

Now let $n>i$ and assume that $\mathcal{P}_{i,n-1}$ is proper in
$\mathcal{D}_{n-1}$.
For the sake of contradiction, suppose that $\mathcal{P}_{i,n}$ is not
proper in $\mathcal{D}_n$ so that it contains the tensor unit
$\mathbbold{1}_n$.
Then $\mathbbold{1}_{n-1}=\mathbbold{1}_n\otimes\mathbbold{1}_{n-1}$ lies
in $\mathcal{P}_{i,n-1}$ by Lemma \ref{lemma:imageofbasechange},
a contradiction.
\end{proof}

We are only now forced to specialise to the case $G=C_p$.

\begin{lemma}\label{lemma:maximalideal}
Each $\mathcal{P}_{i,n}$ is a maximal tensor ideal in
$\mathcal{D}_n$.
\end{lemma}

\begin{proof}
We fix $i$ and proceed by induction on $n$. For the base case $n=i$
we need to show that
$\thick_{\mathcal{D}_i}^\otimes(W_1,\ldots,W_{i-1})=\mathcal{D}_{i-1}$
is maximal in $\mathcal{D}_i$.
Recall that the thick tensor ideals in $\mathcal{D}_i$ containing
$\mathcal{D}_{i-1}$ are in bijection with those in the quotient
$\mathcal{D}_i/\mathcal{D}_{i-1}$.
By the discussion in Section \ref{sec:localisation}, that quotient is
tensor equivalent to $\mathcal{K}_i$, which in turn is equivalent to
\[
  \mathcal{D}_1=\strel kC_p=\stmod kC_p
\]
by Theorem \ref{thm:equivalence}.
It is known that the rightmost category has precisely two tensor ideals,
namely the zero ideal and the entire category.
It follows that the only tensor ideal in $\mathcal{D}_i$ properly
containing $\mathcal{D}_{i-1}$ is $\mathcal{D}_i$ itself,
so $\mathcal{D}_{i-1}$ is maximal as claimed.

Now let $n>i$, assume that $\mathcal{P}_{i,n-1}$ is maximal in
$\mathcal{D}_{n-1}$ and choose $X\notin\mathcal{P}_{i,n}$.
Tensoring the short exact sequence of Lemma \ref{lemma:themagicsequence}
with $X$ produces a triangle
\[
\xymatrix{
X\otimes\mathbbold{1}_{n-1} \ar@{>}[r] & X\otimes W_n \ar@{>}[r] &
\Sigma X \ar@{>}[r] &
}
\]
in $\mathcal{D}_n$.
The middle term lies in $\mathcal{P}_{i,n}$, but the right hand term does not.
This implies that $X\otimes\mathbbold{1}_{n-1}$ does not lie in
$\mathcal{P}_{i,n}$.
In particular, it cannot lie in $\mathcal{P}_{i,n-1}$ since the latter is
con-
tained in the former.
By the inductive hypothesis on maximality, this means that
\[
\mathbbold{1}_{n-1}\in
\thick_{\mathcal{D}_{n-1}}^\otimes
(\{W_1,\ldots,W_{n-1},X\otimes\mathbbold{1}_{n-1}\}\setminus\{W_i\}).
\]
Now consider the triangle
\[
\xymatrix{
\mathbbold{1}_{n-1} \ar@{>}[r] & W_n \ar@{>}[r] &
\Sigma\mathbbold{1}_n \ar@{>}[r] &
}
\]
in $\mathcal{D}_n$ induced by the short exact sequence of
Lemma \ref{lemma:themagicsequence}.
By the above remarks, the left two terms lie in
$\thick_{\mathcal{D}_n}^\otimes(\{W_1,\ldots,W_n,X\}\setminus\{W_i\})$,
whence so does the right hand term.
In other words
\[
\mathbbold{1}_n\in\thick_{\mathcal{D}_n}^\otimes
(\{W_1,\ldots,W_n,X\}\setminus\{W_i\}), 
\]
proving that $\mathcal{P}_{i.n}$ is maximal.
\end{proof}

It now follows from \cite[Proposition 2.3]{Balmer:2005a} that each
$\mathcal{P}_{i,n}$ is a prime tensor ideal, i.e., gives a point in
$\Spc\mathcal{D}_n$. 

\begin{lemma}\label{lemma:minimalprimes}
Each $\mathcal{P}_{i,n}$ is a minimal prime.
\end{lemma}

\begin{proof}
Suppose there exists a prime ideal $\mathcal{P}$ properly contained in
$\mathcal{P}_{i,n}$.
Then there is an object in the set $\{W_1,\ldots,W_n\}\setminus\{W_i\}$
not contained in $\mathcal{P}$.
Since $\mathcal{P}$ is prime, Proposition \ref{prop:omegasgenerate}
forces us to have $W_i\in\mathcal{P}$.
But this implies that $W_i\in\mathcal{P}_{i,n}$,
so $\{W_1,\ldots,W_n\}\subseteq\mathcal{P}_{i,n}$.
Proposition \ref{prop:omegasgenerate} now tells us that
$\mathcal{P}_{i,n}=\mathcal{D}_n$, contradicting
Lemma \ref{lemma:idealsareproper}.
\end{proof}

We are now ready to give a direct computation of the spectrum,
verifying Theorem \ref{thm:maintheorem} in this case.

\begin{theorem}
  The prime ideal spectrum of $\mathcal{D}_n=\strel\, R_nC_p$ is a
  disjoint union of $n$ points.
\end{theorem}

\begin{proof}
If $\mathcal{P}$ is an element of $\Spc\mathcal{D}_n$, then by
Proposition \ref{prop:omegasgenerate} there is an integer $1\leq i\leq n$
such that $\{W_1,\ldots,W_n\}\setminus\{W_i\}\subseteq\mathcal{P}$.
We then have $\mathcal{P}_{i,n}\subseteq\mathcal{P}$.
Since $\mathcal{P}_{i,n}$ is maximal, this implies that
$\mathcal{P}=\mathcal{P}_{i,n}$.
It follows that
$\Spc\mathcal{D}_n=\{\mathcal{P}_{1,n},\ldots,\mathcal{P}_{n,n}\}$
as a set.
We now recall from \cite[Proposition 2.9]{Balmer:2005a} that
the closed points in the prime ideal spectrum are precisely the
\emph{minimal} primes.
Lemma \ref{lemma:minimalprimes} therefore informs us that each point
$\mathcal{P}_{i,n}$ is closed in the topology of $\Spc\mathcal{D}_n$.
\end{proof}

We now make a few observations based on and related to the theorem;
all of the statements are more or less trivial consequences of what we
have already done, but are perhaps worth making explicit.

\begin{corollary}
Each $W_i=\Omega_i^{-1}k$ is supported at a single point, namely
\[
\supp(W_i)=\{\mathcal{P}_{i,n}\}.
\]
\end{corollary}
\begin{proof}
Given the computation of the spectrum and the definition of support,
we have
\[
\supp(W_i)=\{\mathcal{P}\in\Spc\mathcal{D}_n\mid W_i\notin\mathcal{P}\}=
\{\mathcal{P}_{i,n}\}. \qedhere
\]
\end{proof}

\begin{corollary}
The base change functor
\[
\xymatrix{
-\otimes\mathbbold{1}_{m}\colon\mathcal{D}_n \ar@{>}[r] & \mathcal{D}_m
}
\]
induces an embedding $\Spc\mathcal{D}_m\hookrightarrow\Spc\mathcal{D}_n$
with image $\{\mathcal{P}_{i,n}\mid 1\leq i \leq m\}$.
\end{corollary}

\begin{proof}
Letting $\phi^*$ denote the base change functor,
it is a general fact (see \cite[Proposition 3.11]{Balmer:2005a}) that
$\Spc(\phi^*)$ is a homeomorphism onto its image.
The latter set consists precisely of those prime
ideals that contain the kernel of $\phi^*$.
It follows by repeated applications of Lemma \ref{lemma:kernelgenerated} that
the kernel of $\phi^*$ is generated by $\{W_i\mid m+1\leq i \leq n\}$.
The result follows immediately.
\end{proof}

In the above corollary, one might also consider the corresponding
base change functors having source $\strel SC_p$.
The result remains true, except for the description of the image.
By invoking \cite[Theorem A.5]{Benson/Iyengar/Krause:2013a},
one may write
\[
\Spc(\strel SC_p)\cong Z_n \coprod \Spc\mathcal{D}_n
\]
for all $n\geq 1$, where $Z_n$ is the spectrum of the kernel.
Thus for any $n\geq 1$ we can find a disjoint union of $n$ points as
an open and closed subspace of the spectrum of $\strel SC_p$.
It would be interesting to fully understand the space $\Spc(\strel SC_p)$ and,
in particular, how the spectra of the $\mathcal{D}_n$ sit inside of it.

\bibliographystyle{amsplain}
\bibliography{repcoh}

\end{document}